\documentclass{svproc}
\usepackage{url}

\usepackage[T1]{fontenc}
\usepackage{lmodern}
\usepackage{amsmath,amssymb}
\usepackage{hyperref}
\usepackage{microtype}
\usepackage{hyperref}
\usepackage{graphicx}
\usepackage{enumitem}
\usepackage{microtype}

\newtheorem{observation}{Observation}

\begin{document}
\mainmatter              % start of a contribution
\title{On ordered Ramsey numbers of matchings versus triangles}
\titlerunning{On ordered Ramsey numbers of matchings versus triangles
}  % abbreviated title (for running head)
\author{Martin Balko \and Marian Poljak}
\authorrunning{{Martin Balko and Marian Poljak}} % abbreviated author list (for running head)
%\tocauthor{Martin Balko and  Marian Poljak}
%
\institute{
Department of Applied Mathematics, Faculty of Mathematics and Physics, Charles University, Czech Republic\\
\email{\texttt{balko@kam.mff.cuni.cz}, \texttt{marian@kam.mff.cuni.cz}}
%\\ WWW home page: \texttt{https://kam.mff.cuni.cz/\homedir balko/}
}

\maketitle              % typeset the title of the contribution

\begin{abstract}
For graphs $G^<$ and $H^<$ with linearly ordered vertex sets, the \emph{ordered Ramsey number} $r_<(G^<,H^<)$ is the smallest positive integer $N$ such that any red-blue coloring of the edges of the complete ordered graph $K^<_N$ on $N$ vertices contains either a blue copy of $G^<$ or a red copy of $H^<$.
Motivated by a problem of Conlon, Fox, Lee, and Sudakov (2017), we study the numbers $r_<(M^<,K^<_3)$ where $M^<$ is an ordered matching on $n$ vertices.

We prove that almost all $n$-vertex ordered matchings $M^<$ with interval chromatic number 2 satisfy $r_<(M^<,K^<_3) \in \Omega((n/\log n)^{5/4})$ and $r_<(M^<,K^<_3) \in O(n^{7/4})$, improving a recent result by Rohatgi (2019).
We also show that there are $n$-vertex ordered matchings $M^<$ with interval chromatic number at least 3 satisfying $r_<(M^<,K^<_3) \in \Omega((n/\log n)^{4/3})$, which asymptotically matches the best known  lower bound on these off-diagonal ordered Ramsey numbers for general $n$-vertex ordered matchings.

\end{abstract}

\section{Introduction}

For graphs $G$ and $H$, their \emph{Ramsey number} $r(G,H)$ is the smallest positive integer $N$ such that any coloring of the edges of $K_N$ with colors red and blue contains either $G$ as a subgraph with all edges colored red or $H$ as a subgraph with all edges colored blue.
The existence of these numbers was proved by Ramsey~\cite{ramseyprvni} and later independently by Erd\H{o}s and Szekeres~\cite{erdosprvni}.
The study of the growth rate of these numbers with respect to the number of vertices of $G$ and $H$ is a classical part of combinatorics and plays a central role in Ramsey theory.

Motivated by several applications from discrete geometry and extremal combinatorics, various researchers~\cite{balko,bjv16,cp02,conlon,msw15,rohatgi} started investigating an ordered variant of Ramsey numbers.
An \emph{ordered graph} $G^<$ is a graph $G$ together with a linear ordering of its vertex set.
An ordered graph $G^<$ is an \emph{ordered subgraph} of another ordered graph $H^<$ if $G$ is a subgraph of $H$ and the vertex ordering of $G$ is a suborder of the vertex ordering of $H$.
The \emph{ordered Ramsey number} $r_<(G^<,H^<)$ of ordered graphs $G^<$ and $H^<$ is the smallest positive integer $N$ such that any coloring of the edges of the complete ordered graph $K^<_N$ on $N$ vertices with colors red and blue contains either $G^<$ as an ordered subgraph with all edges red or $H^<$ as an ordered subgraph with all edges blue.

It is known that the ordered Ramsey numbers always exist and that they can behave very differently from the unordered Ramsey numbers.
For example, there are ordered \emph{matchings} $M^<$ (that is, 1-regular ordered graphs) on $n$ vertices for which $r_<(M^<,M^<)$ grows superpolynomially in $n$, in particular, we have $r_<(M^<,M^<) \in 2^{\Omega(\log^2{n}/\log{\log{n}})}$~\cite{balko,conlon} while $r(G,G)$ is linear for all graphs $G$ with bounded maximum degrees~\cite{crst83}.
The superpolynomial bound obtained for ordered matchings is almost tight for sparse graphs as, for every fixed $d \in\mathbb{N}$, every $d$-degenerate ordered graph $G^<$ on $n$ vertices satisfies $r_<(G^<,G^<) \in 2^{O(\log^2{n})}$~\cite{conlon}.

One of the most interesting cases for ordered Ramsey numbers is the study of the growth rate of $r_<(M^<,K^<_3)$ where $M^<$ is an ordered matching on $n$ vertices as this is one of the first non-trivial cases where the exact asymptotics is not known.
Conlon, Fox, Lee, and Sudakov~\cite{conlon}  observed that the classical bound $r(K_n,K_3) \in O(n^2/\log{n})$ immediately gives $r_<(M^<,K^<_3) \in O(n^2/\log{n})$.
On the other hand, they showed that there exists a positive constant $c$ such that, for all even positive integers $n$, there is an ordered matching $M^<$ on $n$ vertices with
\begin{equation}
\label{cflsLowerBound}
r_<(M^<,K^<_3) \geq  c\left(\frac{n}{\log n}\right)^{4/3}.
\end{equation}
Conlon, Fox, Lee, and Sudakov expect that the upper bound $r_<(M^<,K^<_3) \leq O(n^2/\log{n})$ is far from optimal and posed the following open problem~\cite{conlon}, which is also mentioned in a survey on recent developments in graph Ramsey theory~\cite{recent_developments}.

\begin{problem}[\cite{conlon}]
\label{cflsProblem}
Does there exist an $\epsilon > 0$ such that for any ordered matching $M^<$ on $n$ vertices $r_<(M^<,K^<_3) \in O(n^{2-\varepsilon})$?
\end{problem}

Problem~\ref{cflsProblem} is one of the most important questions in the theory of ordered Ramsey numbers as in order to get a subquadratic upper bound on $r_<(M^<,K^<_3)$ one has to be able to employ the sparsity of $M^<$ since the bound $r(K_n,K_3) \in O(n^2/\log{n})$ is asymptotically tight by a famous result of Kim~\cite{odhadtrojuhelnik2}.
Being able to use the sparsity of $M^<$ and thus distinguish $M^<$ from $K^<_n$ could help in numerous problems on ordered Ramsey numbers.
However, this is difficult as some ordered matchings $M^<$ can be used to approximate the behavior of complete graphs, which is the reason why the numbers $r_<(M^<,M^<)$ can grow superpolynomially. 

Some partial progress on Problem~\ref{cflsProblem} was recently made by Rohatgi~\cite{rohatgi} who considered ordered matchings with bounded interval chromatic number.
The \emph{interval chromatic number} $\chi_<(G^<)$ of an ordered graph $G^<$ is the minimum number of intervals the vertex set of $G^<$ can be partitioned into so that there is no edge of~$G^<$ with both vertices in the same interval.
The interval chromatic number can be understood as an analogue of the chromatic number for ordered graphs as, for example, there is a variant of the Erd\H{o}s--Stone--Simonovits theorem for ordered graphs~\cite{pachTardos06} that is expressed in terms of the interval chromatic number.

Rohatgi~\cite{rohatgi} showed that the subquadratic bound on $r_<(M^<,K^<_3)$ holds for almost all ordered matchings with interval chromatic number 2 by proving the following result.

\begin{theorem}[\cite{rohatgi}]
\label{thm-rohatgi}
There is a constant $c$ such that for every even positive integer $n$, if an ordered matching $M^<$ on $n$
vertices with $\chi_<(M^<)=2$ is picked uniformly at random, then with high probability
\[r_<(M^<,K^<_3) \leq cn^{24/13}.\]
\end{theorem}

Motivated by Problem~\ref{cflsProblem}, we study the numbers $r_<(M^<,K^<_3)$ for ordered matchings with bounded interval chromatic number.
We strengthen some bounds by Rohatgi~\cite{rohatgi} and by Conlon, Fox, Lee, and Sudakov~\cite{conlon}, obtaining a new partial progress on Problem~\ref{cflsProblem}.

From now on, we omit floor and ceiling signs whenever they are not essential.
For a positive integer $n$, we use $[n]$ to denote the set $\{1,\dots,n\}$.
All logarithms in this paper are base 2.

\section{Our results}

We try to tackle the first non-trivial instance of Problem~\ref{cflsProblem} by considering the typical behavior of the numbers $r_<(M^<,K^<_3)$ for ordered matchings with interval chromatic number 2.
As far as we know, there is no non-trivial lower bound in this case.
In his paper, Rohatgi~\cite{rohatgi} mentions the problem of obtaining lower bounds that would come closer to the upper bound from Theorem~\ref{thm-rohatgi}.
As our first result, we prove the first superlinear lower bound for this case.

\begin{theorem}
\label{jumbled_bipartite_lower_bound}
There exists a positive constant $c$ such that, for all even positive integers $n$, if an ordered matching $M^<$ on $n$
vertices with $\chi_<(M^<)=2$ is picked uniformly at random, then with high probability 
\[r_<(M^<, K_3^<) \geq c\left(\frac{n}{\log n}\right)^{5/4}.\]
\end{theorem}

We also show that this lower bound can be improved for ordered matchings $M^<$ with $\chi_<(M^<) > 2$.

\begin{theorem}
\label{jumbled_k-partite_lower_bound}
For every integer $k \geq 3$, there exists a positive constant $c=c(k)$ such that, for all even positive integers $n$, there exists an ordered
matching $M^<$ on $n$ vertices with $\chi_<(M^<) =k$ satisfying
\[r_<(M^<, K_3^<) \geq c\left(\frac{n}{\log n}\right)^{4/3}.\]
\end{theorem}

Note that the lower bound from Theorem~\ref{jumbled_k-partite_lower_bound} asymptotically matches the bound~\eqref{cflsLowerBound} by Conlon, Fox, Lee, and Sudakov~\cite{conlon}.
Thus, the best known lower bound on $r_<(M^<,K^<_3)$ for general ordered matchings $M^<$ can be obtained also for ordered matchings with bounded interval chromatic number as long as this number is at least $3$.
The proofs of Theorems~\ref{jumbled_bipartite_lower_bound} and~\ref{jumbled_k-partite_lower_bound} are probabilistic and are based on ideas used by Conlon, Fox, Lee, and Sudakov~\cite{conlon}.

Rohatgi~\cite{rohatgi} was also interested in determining how far from the truth the exponent $24/13$ from Theorem~\ref{thm-rohatgi} is.
We narrow the gap there by providing the following upper bound that strengthens Theorem~\ref{thm-rohatgi}.

\begin{theorem}
\label{thm-upperBound}
There is a constant $c$ such that for every even positive integer $n$, if an ordered matching $M^<$ on $n$
vertices with $\chi_<(M^<)=2$ is picked uniformly at random, then with high probability
\[r_<(M^<,K^<_3) \leq cn^{7/4}.\]
\end{theorem}

Note that the difference between the exponent in the lower bound from Theorem~\ref{jumbled_bipartite_lower_bound} and the exponent in the upper bound from Theorem~\ref{thm-upperBound} is exactly $1/2$.

\section{Open problems}

Problem~\ref{cflsProblem} still remains wide open, but there are many interesting intermediate questions that one could try to tackle.
The following interesting conjecture was posed by Rohatgi~\cite{rohatgi}.

\begin{conjecture}[\cite{rohatgi}]
\label{conj-rohatgi}
For every integer $k \geq 2$, there is a constant $\varepsilon = \varepsilon(k) > 0$ such that 
\[r_<(M^<,K^<_3) \in O(n^{2-\varepsilon})\]
for almost every ordered matching $M^<$ on $n$ vertices with $\chi_<(M^<) = k$.
\end{conjecture}

It follows from Theorem~\ref{thm-upperBound} that $\varepsilon(2) \geq 1/4$.
The conjecture is open for all cases with $k \geq 3$. 
Our results suggest that $\varepsilon(2) > \varepsilon(3)$ might hold.

Concerning the ordered matchings $M^<$ with interval chromatic number 2, even in this case the growth rate of $r_<(M^<,K^<_3)$ is not understood, so we pose the following weaker version of Problem~\ref{cflsProblem}.

\begin{conjecture}
There exists an $\epsilon > 0$ such that for any ordered matching $M^<$ on $n$ vertices with $\chi_<(M^<)=2$ we have $r_<(M^<,K^<_3) \in O(n^{2-\varepsilon})$.
\end{conjecture}

In this paper, we considered the variant of this problem for random ordered matchings with interval chromatic number 2, but there is still a gap between our bounds.
It would be very interesting to close it.

\begin{problem}
\label{prob-random}
What is the growth rate of $r_<(M^<,K^<_3)$ for uniform random ordered matchings $M^<$ on $n$ vertices with $\chi_<(M^<)=2$?
\end{problem}

It follows from our results that the answer to Problem~\ref{prob-random} lies somewhere between $\Omega((n/\log{n})^{5/4})$ and $O(n^{7/4})$.
We do not know which of these bounds is closer to the truth.

\section{Proof of Theorem~\ref{jumbled_bipartite_lower_bound}}
\label{sec-bipartite}

Here, we show that for all positive integers $n$, if an ordered matching $M^<$ on $2n$ vertices with $\chi_<(M^<)=2$ is picked uniformly at random, then with high probability 
$r_<(M^<, K_3^<) \in \Omega\left(\frac{n}{\log n}\right)^{5/4}$.
The proof is carried out using a similar probabilistic argument used by Conlon, Fox, Lee, and Sudakov~\cite{conlon} and is based on the Lov\'{a}sz local lemma.

We use the following result about the edge-density between disjoint intervals in a random ordered matching with interval chromatic number 2.

\begin{lemma}
\label{lem-densityRandom}
Let $M^<$ be a uniform random ordered matching on $[2n]$ satisfying $\chi_<(M^<)=2$.
Then, asymptotically almost surely, $M^<$ contains an edge between any two intervals $I \subseteq [n]$ and $J \subseteq \{n+1,\dots,2n\}$, each of length at least $2\sqrt{n\log n}$, and at most $12s\sqrt{\log{n}/n}$ edges between any two disjoint intervals of lengths at most $2\sqrt{n\log n}$ and $s \geq 2\sqrt{n\log{n}}$, respectively.
\end{lemma}
\begin{proof}
For sets $A \subseteq [n]$ and $B \subseteq \{n+1,\dots,2n\}$ with $|A|=t=|B|$, the probability that $M^<$ has no edge between $A$ and $B$ is at most
\[\left(\frac{n-t}{n}\right)\left(\frac{n-t-1}{n-1}\right) \cdots \left(\frac{n-2t+1}{n-t+1}\right) \leq \left(\frac{n-t}{n}\right)^t \leq e^{-t^2/n}\]
where we used the inequalities $\frac{n-t-i}{n-i} < \frac{n-t}{n}$ for every $i>0$ and $1-x \leq e^{-x}$ for every $x \in \mathbb{R}$.
There are at most $n^2$ intervals $I \subseteq [n]$ and $J \subseteq \{n+1,\dots,2n\}$, each of size $t$, and thus the probability that there there are two such intervals with no edge of $M^<$ between them is at most $n^2e^{-t^2/n}$.
This probability goes to 0 with increasing $n$ for $t \geq \frac{3}{2}\sqrt{n\log n}$.
Thus, it suffices to take $t$ as an integer between $\frac{3}{2}\sqrt{n\log n}$ and $2\sqrt{n\log n}$.

Now, consider disjoint subsets $C$ and $D$ of $[2n]$ with $|C| = 2\sqrt{n\log n}$ and $|D|=s \geq 2\sqrt{n\log n}$.
Set $r = 12s\sqrt{\log{n}/n}$.
We show that asymptotically almost surely there are at most $r$ edges of $M^<$ between $C$ and $D$.
This is trivial for $r > 2\sqrt{n\log{n}}$ as there are always at most $|C|= 2\sqrt{n\log{n}}$ edges of $M^<$ between $C$ and $D$.
Thus, we assume $r \leq 2\sqrt{n\log n}$.
Then, the probability that there are $r$ edges of $M^<$ between $C$ and $D$ is at most
\begin{align*}
\binom{2\sqrt{n\log n}}{r}\left(\frac{s}{n}\right)\left(\frac{s-1}{n-1}\right)\cdots\left(\frac{s-r+1}{n-r+1}\right) &\leq \left(\frac{2es\sqrt{n\log n}}{rn}\right)^r \\
&\leq \left(\frac{6s\sqrt{\log n}}{r\sqrt{n}}\right)^r
\end{align*}
as the $i$th edge of such $r$ edges has the other vertex in $D$ with probability $\left(\frac{s-i+1}{n-i+1}\right)$.
The remaining edges can be assigned arbitrarily.
There are at most $n^2$ pairs of disjoint intervals $I$ and $J$ with $|I| = 2\sqrt{n\log n}$ and $|J|=s$ and thus the probability that there there are two such intervals with at least $r$ edges of~$M^<$ between them is at most $n^2\left(\frac{6s\sqrt{\log n}}{r\sqrt{n}}\right)^r$.
Since $s \geq 2\sqrt{n\log{n}}$ we then have $r \geq 24\log{n}$.
Thus, since also $\frac{6s\sqrt{\log n}}{r\sqrt{n}} \leq 1/2$ by the choice of $r$, the upper bound goes to zero with increasing $n$. \qed 
\end{proof}

We note that there is an explicit construction of an ordered matching $M^<_t$ on $2t^2$ vertices that satisfies a similar statement with intervals of size only $t$; see Section~\ref{sec-k-partite}.

The key ingredient in our probabilistic argument is the famous Lov\'{a}sz local lemma, see~\cite{alon16} for example.
We now recall its statement.
 
\begin{lemma}[The Lov\'{a}sz local lemma] 
\label{lem-LLL}
Let $\{A_1, \ldots, A_n\}$ be a finite set of events in a probability space. 
A directed graph $D=(V,E)$ is the \emph{dependency graph} of $A_1,\dots,A_n$ if each  event $A_i$ is mutually independent of all the events from $\{A_j \colon (i,j) \notin E\}$.
Let $x_1,\dots,x_n$ be real numbers such that $0 \leq x_i < 1$ and $\Pr[A_i] \leq x_i \prod_{(i,j) \in E} (1-x_j)$ for every $i \in [n]$.
Then, 
\[\Pr\left[\overline{A_1} \cap \cdots \cap \overline{A_n} \right] \geq \prod_{i=1}^n (1-x_i).\] 
In particular, the probability that none of the events $A_1,\dots,A_n$ hold is positive.
\end{lemma}

We apply the Lov\'{a}sz local lemma to prove the following auxiliary result, which is also used in the proof of Theorem~\ref{jumbled_k-partite_lower_bound}.
For positive integers $r,s$, we use $K^<_{r,s}$ to denote the ordered complete bipartite graph where the color classes of sizes $r$ and $s$ form consecutive intervals in this order.
\begin{lemma}
\label{lem-lovaszApplication}
Let $\alpha, \beta, \gamma >0$ and $\delta \geq 0$ be real numbers satisfying the following three inequalities:
$\alpha + \beta + \gamma - \delta \leq \frac{3}{2}$, $\beta \leq 2 \gamma$, and $\alpha +\gamma \leq 1$.
For a sufficiently large integer $n$, let $\mathcal{G}$ be a family of ordered graphs, each on $n^\beta$ vertices and with $40n^{\frac{3}{2}-\alpha+\delta} \log n$ edges, and assume that $|\mathcal{G}| \leq e^{n^{\beta}\log n}$.
Then, there is a red-blue coloring $\chi$ of the edges of $K^<_{n^\beta}$ such that the following three conditions are satisfied:

\begin{enumerate}[label=(\alph*)]
\item there is no blue triangle in $\chi$,
\item there is no red copy of any ordered graph from $\mathcal{G}$ in $\chi$, and
\item there is no red copy of $K^<_{10n^{1-\alpha}\log n,10n^{1-\alpha}\log n}$ in $\chi$.
\end{enumerate}
\end{lemma}
\begin{proof}
We color each edge of $K^<_{n^\beta}$ independently at random blue with probability $\frac{1}{2n^{\gamma}}$ and red with probability $1- \frac{1}{2n^{\gamma}}$.
Let $P_i$ be the events corresponding to blue triangles in our random coloring, let $Q_i$ be the events corresponding to red ordered graphs from $\mathcal{G}$, and $R_i$ the events corresponding to the red ordered complete bipartite graphs as in the statement of the lemma. 
We denote the index sets of the events $P_i$, $ Q_i$, and $R_i$ by $I_P$, $I_Q$, and $I_R$, respectively.
Clearly, we have $|I_P| \leq \binom{n^\beta}{3} \leq n^{3\beta}$ and $|I_Q| = |\mathcal{G}|\leq e^{n^{\beta }\log n}$.
Since $n$ is sufficiently large, we obtain
\begin{align*}
|I_R| &\leq \binom{n^\beta}{20n^{1-\alpha} \log n} \leq \left(\frac{en^\beta}{20n^{1-\alpha}\log n}\right)^{20n^{1-\alpha} \log n} \leq (n^{\alpha + \beta -1})^{20n^{1-\alpha} \log n} \\
&\leq e^{20n^{1-\alpha} \log^2 n}
\end{align*}
as $\alpha+\beta-1\leq 1$ follows from our assumptions $\beta \leq 2\gamma$ and $\alpha+\gamma\leq 1$.

We apply the Lov\'{a}sz local lemma (Lemma~\ref{lem-LLL}) with the events $P_i$, $Q_i$, and $R_i$.
It suffices to verify the conditions of the lemma as then it follows that the probability that none of these events hold is positive. 
That is, there is a coloring $\chi$ satisfying the statement of Lemma~\ref{lem-lovaszApplication}.

We choose $x=\frac{1}{4n^{3\gamma}}$, $y=e^{-2n^{\beta} \log n}$, and $z=e^{-21n^{1-\alpha} \log^2 n}$.
It follows from the choice of $y$ and $z$ and our estimates on $I_Q$ and $I_R$ that $y|I_Q| \in o(1)$ and $z|I_R| \in o(1)$. 
We now verify the conditions of Lemma~\ref{lem-LLL}.

\begin{enumerate}[label=(\roman*)]
\item Events $P_i$:

Each event $P_i$ depends on exactly $3n^\beta$ events $P_j$ and on at most $|I_Q|$ events $Q_j$ and on at most $|I_R|$ events $R_j$.
Thus,
\begin{align*}
x & \prod_{j \in I_P, j \sim i} (1-x) \prod_{j \in I_Q, j \sim i} (1-y) \prod_{j \in I_R, j \sim i} (1-z) \\
& =(1-o(1)) \cdot xe^{-3xn^\beta}e^{-y|I_Q|}e^{-z|I_R|} = (1-o(1)) \cdot \frac{1}{4n^{3\gamma}}e^{-\frac{3}{4}n^{\beta-3\gamma}} \\
&\geq \frac{1}{8n^{3\gamma}}=\Pr[P_i].
\end{align*}
The last inequality holds for a sufficiently large $n$ if and only if $\beta < 3\gamma$, which follows from our stronger assumption $\beta \leq 2\gamma$. 

\item Events $Q_i$:

Every ordered graph corresponding to the event $Q_i$ contains $40n^{\frac{3}{2}-\alpha+\delta} \log n$ edges and thus $Q_i$ depends on at most $40n^{\frac{3}{2}-\alpha+ \beta + \delta} \log n$ events $P_j$. 
It then follows that 
\begin{align*}
y&\prod_{j \in I_P, j \sim i} (1-x) \prod_{j \in I_Q, j \sim i} (1-y) \prod_{j \in I_R, j \sim i} (1-z) \\
&=(1-o(1)) \cdot ye^{-x(40n^{\frac{3}{2}-\alpha+\beta+\delta} \log n)}e^{-y|I_Q|}e^{-z|I_R|} \\
&= (1-o(1))\cdot e^{-2n^{\beta} \log n} e^{-10n^{\frac{3}{2}-\alpha+\beta-3\gamma+\delta}\log n}.
\end{align*}
We want the last expression to be at least $\Pr[Q_i]=(1-\frac{1}{2n^{\gamma}})^{40n^{\frac{3}{2}-\alpha+\delta} \log n}$ which is at most $e^{-20n^{\frac{3}{2}-\alpha-\gamma+\delta}\log n}$.
Therefore, it suffices to show that 
\[(1-o(1)) \cdot e^{-2n^{\beta} \log n} e^{-10n^{\frac{3}{2}-\alpha+\beta-3\gamma+\delta} \log n} \geq e^{-20n^{\frac{3}{2}-\alpha-\gamma+\delta}\log n}.\]
This is true if
\[2n^{\beta}  + 10n^{\frac{3}{2}-\alpha+\beta-3\gamma+\delta} \leq 20n^{\frac{3}{2}-\alpha-\gamma+\delta}\]
and the right hand side grows faster than the left one to beat the $(1-o(1))$-term above.
This follows from our assumptions $\alpha + \beta + \gamma-\delta \leq \frac{3}{2}$ and $\beta \leq 2\gamma$ as then $n^\beta \leq n^{\frac{3}{2}-\alpha-\gamma+\delta}$ and also $n^{\frac{3}{2}-\alpha+\beta-3\gamma+\delta} \leq n^{\frac{3}{2}-\alpha-\gamma+\delta}$.

\item
Events $R_i$:

Every ordered complete bipartite graph corresponding to the event $R_i$ contains $(10n^{1-\alpha} \log n)^2 = 100n^{2-2\alpha} \log^2 n$ edges and thus $R_i$ depends on at most $100n^{2-2\alpha+\beta} \log^2 n$ events $P_i$.
It follows that
\begin{align*}
z&\prod_{j \in I_P, j \sim i} (1-x) \prod_{j \in I_Q, j \sim i} (1-y) \prod_{j \in I_R, j \sim i} (1-z) \\
&=  
 (1-o(1)) \cdot ze^{-x(100n^{2-2\alpha+\beta} \log^2 n)}e^{-y|I_Q|}e^{-z|I_R|} \\
&= (1-o(1)) \cdot e^{-21n^{1-\alpha} \log^2 n} e^{-25n^{2-2\alpha+\beta-3\gamma}\log^2 n}.
\end{align*}

We want the last expression to be at least $\Pr[R_i]=(1-\frac{1}{2n^{\gamma}})^{100n^{2-2\alpha} \log^2 n}$, which is at most $e^{-50n^{2-2\alpha-\gamma}\log^2 n}$. 
That is, it suffices to check that
\[(1-o(1)) \cdot e^{-21n^{1-\alpha} \log^2 n} e^{-25n^{2-2\alpha+\beta-3\gamma}\log^2 n} \geq e^{-50n^{2-2\alpha-\gamma}\log^2 n}.\]
This is true if
\[21n^{1-\alpha}  + 25n^{2-2\alpha+\beta-3\gamma} \leq 50n^{2-2\alpha-\gamma}\]
and the right hand side grows faster than the left one to beat the $(1-o(1))$-term above.
This follows from our assumptions $\alpha + \gamma \leq 1$ and $\beta \leq 2\gamma$ as then $n^{1-\alpha} \leq n^{2-2\alpha-\gamma}$ and $n^{2-2\alpha+\beta-3\gamma} \leq  n^{2-2\alpha-\gamma}$.
\end{enumerate}

Altogether, all conditions of Lemma~\ref{lem-LLL} are satisfied and we obtain the desired coloring $\chi$.  \qed
\end{proof}

We now proceed with the proof of Theorem~\ref{jumbled_bipartite_lower_bound}.
The approach is similar to the one used by Conlon, Fox, Lee, and Sudakov~\cite{conlon}.

\begin{proof}[Proof of Theorem~\ref{jumbled_bipartite_lower_bound}]
Let $M^<$ be a random ordered matching with $\chi_<(M^<)=2$ on $m = 800n\log n$ vertices.
We will prove that 
$n^{5/4}\le r_{<}(M^<, K^<_3)$ for $n$ sufficiently large. 
We choose $\alpha=\frac{3}{4}$, $\beta=\frac{1}{2}$, and $\gamma=\frac{1}{4}$.
Note that this choice of parameters satisfies the conditions in the statement of Lemma~\ref{lem-lovaszApplication} with $\delta=0$.

We set $N=n^{\alpha + \beta} = n^{5/4}$ and we partition $[N]$ into consecutive intervals $V_{1},\dots,V_{n^{\beta}}$, each
of length $n^{\alpha}$.
Let $\Phi$ be the set of injective embeddings of $M^<$ into $[N]$ that respect the vertex ordering of $M^<$.
For $\phi \in \Phi$, let $G^<(\phi)$ be the ordered graph on the vertex set $[n^{\beta}]$ with an edge between $i$ and
$j$ if and only if there is an edge of $M^<$ between $\phi^{-1}(V_{i})$ and $\phi^{-1}(V_{j})$.
Let 
\[
\mathcal{H}=\{G^<(\phi)\,:\,\phi\in\Phi,\, |E(G^<(\phi))|\ge40n^{\frac{3}{2}-\alpha}\log n\}.
\]

Since the mappings in $\Phi$ respect the order of the vertices of $M^<$, any ordered graph $G^<(\phi)$
is determined by the last vertex in $\phi^{-1}(V_{i})$ for every
$i\in[n^{\beta}]$.
Therefore, for $n$ sufficiently large, we obtain
\[
|\mathcal{H}|\le \binom{m+n^\beta}{n^\beta} \leq \left(\frac{e(m+n^{\beta})}{n^{\beta}}\right)^{n^\beta} = \left(e(800n^{1-\beta}\log n+1)\right)^{n^{\beta}}\le e^{n^{\beta}\log n}.
\]
Let $\mathcal{G}$ be the class of ordered graphs such that for every $H^< \in \mathcal{H}$ there is $G^<\in \mathcal{G}$ with exactly $40n^{\frac{3}{2}-\alpha}\log{n}$ edges such that $G^<$ is an ordered subgraph of $H^<$.
Note that we can choose $\mathcal{G}$ so that $|\mathcal{G}| \leq |\mathcal{H}|$.

Applying Lemma~\ref{lem-lovaszApplication} to $\mathcal{G}$ with our choice of $\alpha$, $\beta$, $\gamma$, and $\delta$, we obtain a red-blue coloring $\chi'$ of the edges of $K^<_{n^\beta}$ that avoids a blue triangle, a red copy of any ordered graph from $\mathcal{G}$, and a red copy of $K^<_{10n^{1-\alpha} \log n,10n^{1-\alpha} \log n}$.

Let $\chi$ be the red-blue coloring of the edges of the ordered complete graph on $[N]$ where we color all edges between
$V_{i}$ and $V_{j}$ with color $\chi'(i,j)$ for all $i,j\in[n^{\beta}]$.
We color all edges within the sets $V_{i}$ red.
Note that $\chi$ contains no blue triangle, since $\chi'$ does not contain
a blue triangle.

Suppose for contradiction that for some $\phi\in\Phi$, the ordered matching $\phi(M^<)$ is a red copy of $M^<$ in $\chi$.
We use $P_1$ and $P_2$ to denote the left and the right color class of $\phi(M^<)$, respectively, each of size $m/2=400n\log n$. 
Let $W_{i}=V(\phi(M^<))\cap V_{i}$ for each $i$ and let $S\subseteq[n^{\beta}]$ be the set of indices $i$ for which $|W_{i}|\le 2\sqrt{m \log{m}}$.
We set $L=[n^{\beta}]\setminus S$. 

By Lemma~\ref{lem-densityRandom}, for any pair of indices $i,j\in L$ with $W_i \subseteq P_1$ and $W_j \subseteq P_2$, there is an edge of $\phi(M^<)$ between $W_{i}$ and $W_{j}$ since $|W_{i}|,|W_{j}|> 2\sqrt{m\log{m}}$ and $M^<$ is a random ordered matching with $\chi_<(M^<)=2$. 
Then, $\chi'(i,j)$ is red as all edges of $\phi(M^<)$ are red in $\chi$.
Thus, if there are at least $10n^{1-\alpha} \log n$ sets $W_i$ with $i \in L$ inside each of the two color classes $P_1$ and $P_2$ of $\phi(M^<)$, then we have a red copy of $K_{10n^{1-\alpha} \log n,10n^{1-\alpha} \log n}^<$ in $\chi'$.
This is impossible by the choice of~$\chi'$.

Hence, one of the color classes of $\phi(M^<)$ contains less than $10n^{1-\alpha} \log n$ sets $W_i$ with $i \in L$.
By symmetry, we can assume that it is the color class $P_1$.
Since the size of any set $W_i$ is at most $|V_i| = n^\alpha$, each set $W_i$ is incident to at most $n^\alpha$ edges of the ordered matching $\phi(M^<)$.
Overall, all sets $W_i  \subseteq P_1$ with $i \in L$ are incident to at most $10n\log n$ edges of $\phi(M^<)$.
Therefore, there are at least $390n\log n$ edges of $\phi(M^<)$ incident to sets $W_i \subseteq P_1$ with $i \in S$. 

Consider $W_i \subseteq P_1 $ and $W_j \subseteq P_2$ such that $i \in S$.
We recall that $|W_i| \leq 2\sqrt{m\log{m}}$ and $|W_j| \leq n^\alpha$.
By Lemma~\ref{lem-densityRandom}, there are at most $12n^\alpha \sqrt{\log{m}/m} \leq n^{\alpha-1/2}$ edges of $\phi(M^<)$ between $W_i$ and $W_j$ for $n$ sufficiently large. 
Since there are at least $390n \log n$ edges of $\phi(M^<)$ incident to sets $W_i \subseteq P_1$ with $i \in S$, there are at least 
\[\frac{390n \log n}{n^{\alpha-1/2}}=390n^{\frac{3}{2}-\alpha} \log n > 40n^{\frac{3}{2}-\alpha} \log n\]
red edges in the coloring $\chi'$.
This implies that $G^<(\phi) \in \mathcal{H}$. However, $G^<(\phi)$ has all edges red in the coloring $\chi'$ which contradicts the choice of $\chi'$.

Altogether, there is no red copy of $M^<$ and no blue copy of $K^<_3$ in $\chi$ and thus $r_<(M^<,K^<_3) > N=n^{5/4}$. \qed
\end{proof}

\section{Proof of Theorem~\ref{jumbled_k-partite_lower_bound}}
\label{sec-k-partite}

We show that, for every integer $k \geq 3$ and for all positive integers $n$, there exists an ordered matching $M^<$ on $2n$ vertices with $\chi_<(M^<) =k$ satisfying
$r_<(M^<, K_3^<) \in \Omega\left(\frac{n}{\log n}\right)^{4/3}$.

First, we construct the following auxiliary ordered matching $M^<_t$ with interval chromatic number 2.
For a positive integer $t$, let $[2t^2]$ be the vertex set of $M^<_t$.
We partition the set $[t^2]$ into $t$ consecutive intervals $I_1,\dots,I_t$, each of size $t$ and, similarly, let $J_1,\dots,J_t$ be the partition of the set $\{t^2+1,\dots,2t^2\}$ into $t$ consecutive intervals, each of size $t$.
Then, for all distinct integers $i$ and $j$ with $1 \leq i, j \leq t$ we put an edge between the $j$th vertex of $I_i$ and the $i$th vertex of~$J_j$.
Note that there is exactly one edge between each $I_i$ and $J_j$.
The ordered matching $M^<_t$ then satisfies the following properties.

\begin{observation}
\label{lem-jumbledBipartite}
There is at least one edge of~$M^<_t$ between any two intervals $I \subseteq \cup_{i=1}^tI_i$ and $J \subseteq \cup_{j=1}^tJ_j$, each of length at least $2t$.
\end{observation}
\begin{proof}
The interval $I$ contains some interval $I_i$ and $J$ contains some interval $J_j$, so there is an edge of~$M^<$ between $I$ and $J$.
\qed
\end{proof}

For positive integers $k \geq 3$ and $t$, we now construct the ordered matching $M^<_{k,t}$ on $m = k(k-1)t^2$ vertices that is used in the proof of Theorem~\ref{jumbled_k-partite_lower_bound}.
The main idea is to define $M^<_{k,t}$ as an intertwined union of the ordered matchings $M^<_t$; see Figure~\ref{fig-jumbled} for an illustration with $k=3$ and $t=3$.

The vertex set $[m]$ of $M^<_{k,t}$ is partitioned into consecutive intervals $P_1,\dots,P_k$, each of size $m/k = (k-1)t^2$.
For every $i \in [k]$, the interval $P_i$ is partitioned into consecutive intervals $B_{i,1},\dots,B_{i,(k-1)t}$, each of size $t$.
We call each interval $B_{i,j}$ a \emph{block} of~$M^<_{k,t}$.
For every $i \in [k]$ and $j \in \{0,1,\dots,k-2\}$, let $a_j$ be the $(j+1)$st smallest element of $[k] \setminus \{i\}$ and let $C_{i,a_j}$ be the set of vertices that is the union of the blocks $B_{i,\ell}$ where $\ell$ is congruent to $j$ modulo $k-1$; see Figure~\ref{fig-jumbled}.
We call each set $C_{i,a_j}$ a \emph{superblock} of $M^<_{k,t}$.
Note that the size of each superblock is $t^2$.
We now place the edges so that any pair $C_{i,j}$ and $C_{j,i}$ of superblocks induces a copy $M^<(i,j)$ of the ordered matching $M^<_t$. 

\begin{figure}[ht]
    \centering
    \includegraphics{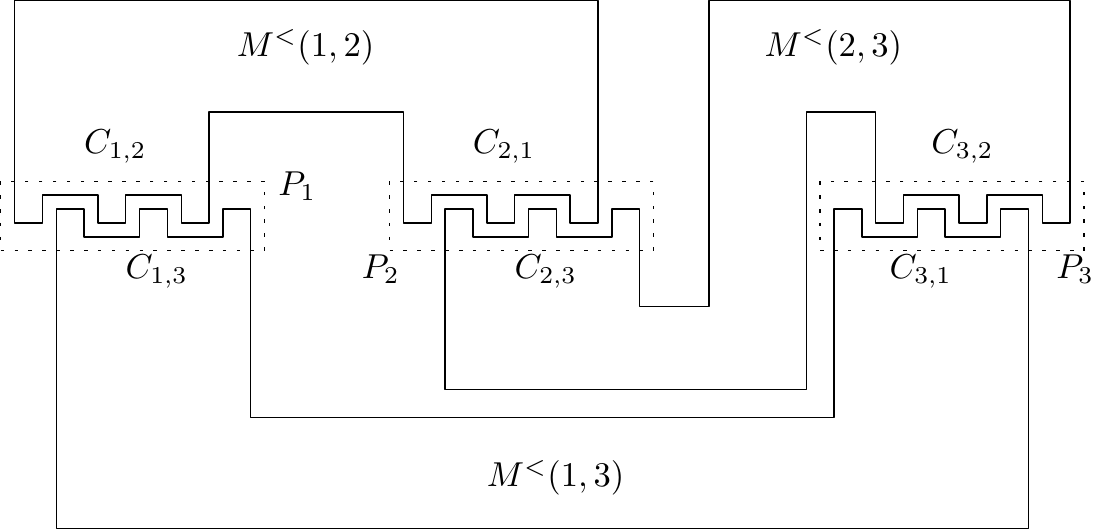}
    \caption{An illustration of the ordered matching $M^<_{k,t}$ for $k=3$ and $t=3$.}
    \label{fig-jumbled}
\end{figure}

Observe that $\chi_<(M^<_{k,t})=k$ as the sets $P_1,\dots,P_k$ form the color classes of~$M^<_{k,t}$.
We now state the key properties of the ordered matching $M^<_{k,t}$.

\begin{lemma}
\label{lem-jumbled}
Let $i,j \in [k]$ be two distinct integers.
For any pair of intervals $I \subseteq P_i$ and $J \subseteq P_j$, each of length at least $2kt$, there is an edge of $M^<_{k,t}$ between $I$ and~$J$.
Moreover, there are at most $(2k+1)^2$ edges between any two disjoint intervals $I' \subseteq P_i$ and $J' \subseteq P_j$, each of size at most $2kt$. 
\end{lemma}
\begin{proof}
First, let $I$ and $J$ be the two intervals from the first part of the statement.
Since $|I| \geq 2kt$, the interval $I$ intersects each superblock $C_{i,j'}$, $j'\ \in [k] \setminus \{i\}$, in an interval of length at least $2t$.
Analogously, $J$ intersects each superblock $C_{j,i'}$, $i'\ \in [k] \setminus \{j\}$, in an interval of length at least $2t$.
Then, by Observation~\ref{lem-jumbledBipartite}, there is an edge of $M^<(i,j) \subseteq M^<_{k,t}$ between the sets $I \cap C_{i,j}$ and $J \cap C_{j,i}$.

Let $I'$ and $J'$ be the intervals from the second part of the statement.
Since $|I'| \leq 2kt$ and since the size of each block of $M^<_{k,t}$ is $t$, the interval $I'$ can intersect at most $2k+1$ blocks.
An analogous claim is true for the interval $J'$.
Since there is at most one edge between any pair of blocks, it follows that there can be at most $(2k+1)^2$ edges of $M^<_{k,t}$ between $I'$ and $J'$.
\qed
\end{proof}

We now proceed with the proof of Theorem~\ref{jumbled_k-partite_lower_bound}.
The proof is similar to the proof of Theorem~\ref{jumbled_bipartite_lower_bound}.

\begin{proof}[Proof of Theorem~\ref{jumbled_k-partite_lower_bound}]
For a given integer $k \geq 3$, we choose $t$ sufficiently large and express the number $m=k(k-1)t^2$ of vertices of $M^<_{k,t}$ as $m = 500k^3 n\log n$ for some positive integer $n$.
We will prove that 
$n^{4/3}\le r_{<}(M^<_{k,t}, K^<_3)$.
We set $\alpha=\frac{2}{3}$, $\beta=\frac{2}{3}$, and $\gamma=\frac{1}{3}$.
Note that this choice of parameters satisfies the conditions in the statement of Lemma~\ref{lem-lovaszApplication} with $\delta =\alpha-1/2=1/6$.

We set $N=n^{\alpha + \beta} = n^{4/3}$ and we partition $[N]$ into consecutive intervals $V_{1},\dots,V_{n^{\beta}}$, each
of length $n^{\alpha}$.
Similarly as before, we let $\Phi$ be the set of injective embeddings of $M^<_{k,t}$ into $[N]$ that respect the vertex ordering of $M^<_{k,t}$.
For $\phi \in \Phi$, let $G^<(\phi)$ be the ordered graph on the vertex set $[n^{\beta}]$ with an edge between $i$ and
$j$ if and only if there is an edge of $M^<_{k,t}$ between $\phi^{-1}(V_{i})$ and $\phi^{-1}(V_{j})$.
We also set 
\[
\mathcal{H}=\{G^<(\phi)\,:\,\phi\in\Phi,\, |E(G^<(\phi))|\ge40n\log n\}.
\]
Analogously as in the proof of Theorem~\ref{jumbled_bipartite_lower_bound}, we obtain $|\mathcal{H}|\le e^{n^{\beta}\log n}$.
Let $\mathcal{G}$ be the class of ordered graphs such that for every $H^< \in \mathcal{H}$ there is $G^<\in \mathcal{G}$ with exactly $40n\log{n}$ edges such that $G^<$ is an ordered subgraph of $H^<$.
Note that we can choose $\mathcal{G}$ so that $|\mathcal{G}| \leq |\mathcal{H}|$.

Applying Lemma~\ref{lem-lovaszApplication} to $\mathcal{G}$ with our choice of $\alpha$, $\beta$, $\gamma$, and $\delta$, we get a red-blue coloring $\chi'$ of the edges of $K^<_{n^\beta}$ that avoids a blue triangle, a red copy of any ordered graph from $\mathcal{G}$, and a red copy of $K^<_{10n^{1-\alpha} \log n,10n^{1-\alpha} \log n}$.

Let $\chi$ be the red-blue coloring of the edges of the ordered complete graph on $[N]$ where we color all edges between
$V_{i}$ and $V_{j}$ with color $\chi'(i,j)$ for all $i,j\in[n^{\beta}]$.
We color all edges within the sets $V_{i}$ red.
Note that $\chi$ contains no blue triangle, since $\chi'$ does not contain
a blue triangle.

Suppose for contradiction that for some $\phi\in\Phi$, the ordered matching $\phi(M^<_{k,t})$ is a red copy of $M^<$ in $\chi$.
We use $\phi(P_1),\dots,\phi(P_k)$ to denote the color classes of $\phi(M^<_{k,t})$.
Let $W_{i}=V(\phi(M^<))\cap V_{i}$ for each $i$ and let $S\subseteq[n^{\beta}]$ be the set of indices $i$ for which $|W_{i}|\le 2kt$.
We set $L=[n^{\beta}]\setminus S$.

By Lemma~\ref{lem-jumbled}, for any pair of indices $i,j\in L$ with $W_i$ and $W_j$ that are contained in different color classes of $\phi(M^<_{k,t})$, there is an edge of $\phi(M^<_{k,t})$ between $W_i$ and $W_j$ since $|W_{i}|,|W_{j}|> 2kt$.
Then, $\chi'(i,j)$ is red as all edges of $\phi(M^<)$ are red in $\chi$.
Thus, if there are two color classes $\phi(P_a)$ and $\phi(P_b)$, each with at least $10n^{1-\alpha} \log n$ sets $W_i$ with $i \in L$, then we have a red copy of $K_{10n^{1-\alpha} \log n,10n^{1-\alpha} \log n}^<$ in $\chi'$.
This is impossible by the choice of~$\chi'$.

Thus, at most one color class of $\phi(M^<_{k,t})$ contains at least $10n^{1-\alpha} \log n$ sets $W_i$ with $i \in L$.
Since $k \geq 3$, there are two color classes of $\phi(M^<_{k,t})$, without loss of generality $\phi(P_1)$ and $\phi(P_2)$, such that each one of them contains less than $10n^{1-\alpha} \log n$ sets $W_i$ with $i \in L$. 
Note that $|\phi(P_1)| = |\phi(P_2)| = m/k = 500k^2n\log{n}$.
Since the size of any set $W_i$ is at most $|V_i| = n^\alpha$, each set $W_i$ is incident to at most $n^\alpha$ edges of the ordered matching $\phi(M^<_{k,t})$.
Overall, all sets $W_i  \subseteq \phi(P_1)\cup \phi(P_2)$ with $i \in L$ are incident to at most $20n\log n$ edges of $\phi(M^<)$.
Therefore, there are at least $(500k^2-20)n\log n \geq  480k^2n\log{n}$ edges of $\phi(M^<)$ incident to sets $W_i \subseteq \phi(P_1) \cup \phi(P_2)$ with $i \in S$. 

Consider $W_i \subseteq \phi(P_1) $ and $W_j \subseteq \phi(P_2)$ such that $i,j \in S$.
We recall that $|W_i|, |W_j| \leq 2kt$.
By Lemma~\ref{lem-jumbled}, there are at most $(2k+1)^2 < 10k^2$ edges of~$\phi(M^<_{k,t})$ between $W_i$ and $W_j$.
Since there are at least $480k^2n \log n$ edges of~$\phi(M^<_{k,t})$ incident to sets $W_i \subseteq \phi(P_1)$ with $i \in S$, there are at least 
\[\frac{480k^2n \log n}{10k^2} > 40n \log n\]
red edges in the coloring $\chi'$.
This implies that $G^<(\phi) \in \mathcal{H}$. However, $G^<(\phi)$ has all edges red in the coloring $\chi'$ which contradicts the choice of $\chi'$.

Altogether, there is no red copy of $M^<$ nor a blue copy of $K^<_3$ in $\chi$ and thus $r_<(M^<,K^<_3) > N=n^{4/3}$. \qed
\end{proof}

\section{Proof of Theorem~\ref{thm-upperBound}}
\label{sec-upperBound}

In this section, we prove the upper bound $r_<(M^<,K^<_3) \in O(n^{7/4})$ for uniform random ordered matchings $M^<$ on $2n$ vertices with $\chi_<(M^<)=2$.
The proof is carried out using a multi-thread scanning procedure whose variants were recently used by Cibulka and Kyn\v{c}l~\cite{cibKyn17}, He and Kwan~\cite{kwan20}, and Rohatgi~\cite{rohatgi}.

First, note that the set of ordered matchings on $2n$ vertices with interval chromatic number $2$ is in one-to-one correspondence with the set of permutations on $[n]$.
Since it is often convenient to work with the permutation corresponding to a given ordered matching $M^<$ on $[2n]$ with $\chi_<(M^<) = 2$, we define the permutation $\pi_{M^<}$ as the permutation on $[n]$ that maps $i$ to $j-n$ for every edge $\{i,j\}$ of $M^<$.
A uniform random ordered matching on $[2n]$ then corresponds to a uniform permutation on $[n]$ selected uniformly at random.

Let $\chi$ be a red-blue coloring of the edges of $K^<_{2N}$ for some positive integer~$N$.
Let $A$ be an $N \times N$ matrix where an entry on position $(i,j) \in [N]\times[N]$ contains the color of the edge $\{i,N+j\}$ in $\chi$.
Note that a red copy of $M^<$ with one color class in~$[N]$ and the other one in $\{N+1,\dots,2N\}$ corresponds to an $n \times n$ submatrix of~$A$ with red entries on positions $(i,\pi_{M^<}(i))$ for $i=1,\dots,n$.

We now describe a procedure that we use to find a red copy of $M^<$ in $\chi$; see Figure~\ref{fig-scanning} for an illustration.
Let $T$ be a positive integer.
We try to find a red copy of $M^<$ in rows $t+1,\dots,t+n$ for every $t \in \{0,1,\dots,T-1\}$.
First, we scan through the row $\pi_{M^<}(1) + t$ of $A$ from left to right until we find a red entry in some
position $(\pi_{M^<}(1) + t, j_1)$. 
For every $i\in\{2,\dots,n\}$, after we have finished scanning through rows $\pi_{M^<}(1)+t,\dots,\pi_{M^<}(i-1)+t$, we scan through the row $\pi_{M^<}(i) + t$ of $A$, starting from column $j_{i-1} + 1$, until we find a red entry in some position $(\pi_{M^<}(i) + t, j_i)$. 

We call this \emph{multi-thread scanning} for $M^<$ and we call the set $Th(t)$ of entries of $A$ that are revealed in step $t$ a \emph{thread}.
Note that a thread $Th(t)$ successfully finds a red copy of $M^<$ if and only if some red copy of $M^<$ lies in the rows $t + 1,\dots, t + n$ of $A$.
Moreover, if the thread $Th(t)$ does not find a red copy of~$M_<$, then it reveals at least $N-n$ blue entries of $A$.

\begin{figure}[ht]
    \centering
    \includegraphics{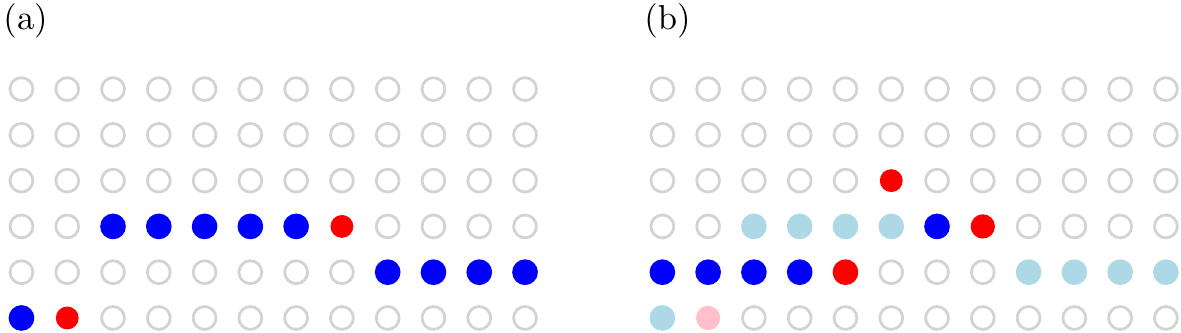}
    \caption{An illustration of the multi-thread scanning procedure for the ordered matching $M^<$ with the corresponding permutation $\pi(M^<) = 132$.
    (a) Thread $Th(0)$ did not find a red copy of $M^<$. (b) Thread $Th(1)$ successfully found a  red copy of $M^<$. The entries whose color was previously revealed by thread $Th(0)$ are denoted by light blue and light red.}
    \label{fig-scanning}
\end{figure}

For a permutation $\pi$ on $[n]$, we say that a subset $C\subseteq[n]$ with $|C|=k$ is a \emph{shift} of another subset $D\subseteq [n]$
in $\pi$ if there is a positive integer $\Delta$ such that $\pi\left(c_{i}\right)=\pi\left(d_{i}\right)+\Delta$ for each $i \in [k]$ where $c_1<\dots<c_k$ and $d_1<\dots<d_k$ are the elements of~ $C$ and $D$, respectively. 
Let $L(\pi)$ be the largest positive integer $k$ for which there are sets $C,D\subseteq [n]$, each of size $k$,
such that $C$ is a shift of $D$.
This notion captures the maximum size of a pattern that a permutation can share with its translation.

We now state the following upper bound on ordered Ramsey numbers of ordered matchings $M^<$ with restricted $L(\pi_{M^<})$ versus triangles, which is used later to derive Theorem~\ref{thm-upperBound}.
A similar result was proved by Rohatgi~\cite{rohatgi}, but it yields asymptotically weaker bounds.

\begin{theorem}
\label{thm-upperBoundStronger}
For a positive integer $n$, let $M^<$ be an ordered matching on $2n$ vertices with $\chi_<(M^<)=2$ and $L(\pi_{M^<}) \leq \ell$.
If $N \geq 4n(\sqrt{n \ell} +1)$, then every red-blue coloring $\chi$ of the edges of $K^<_{2N}$ on $[2N]$ satisfies at least one of the following three claims:
\begin{enumerate}[label=(\alph*)]
\item $\chi$ contains a blue copy of $K^<_3$,
\item $\chi$ contains a red copy of $K^<_{2n}$, or
\item $\chi$ contains  a red copy of $M^<$ between $[N]$ and $\{N+1,\dots,2N\}$.
\end{enumerate}
\end{theorem}

\begin{proof}
Let $\chi$ be a red blue coloring of the edges of $K^<_{2N}$ on $[2N]$.
Suppose for contradiction that $\chi$ satisfies none of the three claims from the statement of the theorem.

Let $A$ be the matrix $N \times N$ matrix where an entry on position $(i,j) \in [N]\times[N]$ contains the color of the edge $\{i,N+j\}$ in $\chi$.
We set $T=\sqrt{n/\ell}$ and we run the multi-thread scanning for $M^<$ in $A$ with $T$ threads.
For every $t \in \{0,1,\dots,T-1\}$, the blue entries from the thread $Th(t)$ intersect each row of $A$ in a set that we call \emph{segment}.
Let $S(t)$ be the set of segments obtained from $Th(t)$.

Observe that each segment forms an interval of blue entries in a row of $A$.
Moreover, each segment has length less than $2n$ as otherwise there is a vertex of~$K_{2N}^<$ incident to at least $2n$ blue edges in $\chi$ and, since there is no blue triangle in $\chi$, the neighborhood of such a vertex induces a red copy of $K^<_{2n}$.
This is impossible by our assumptions on $\chi$.

\begin{claim}
Fix $t$ and $t'$ with $0 \leq t' < t \leq T$.
Assume that $k$ segments in $S(t)$ intersect with some segments from $S(t')$.
Then, $L(\pi_{M^<}) \geq k$.
\end{claim}

Each segment from $S(t)$ intersects at most one segment from $S(t')$ as no two segments from $S(t)$ lie in the same row of $A$ and the same claim is true for segments from $S(t')$.
Moreover, if two segments intersect, then they are contained in the same row of $A$.
Let $s^t_1$ and $s^t_2$ be two segments from $S(t)$ and let $s^{t'}_1$ and $s^{t'}_2$ be two segments from $S(t')$ and assume $s^t_1 \cap s^{t'}_1 \neq \emptyset$ and $s^t_2 \cap s^{t'}_2 \neq \emptyset$.
Then, the columns of $A$ intersected by $s^t_1$ are to the left of the columns intersected by $s^t_2$ if and only if the columns of $A$ intersected by $s^{t'}_1$ are to the left of the columns intersected by $s^{t'}_2$.
This is because no two segments from $S(t)$ intersect the same column from $A$ and the same claim is true for segments from $S(t')$.
Altogether, the $k$ segments from $S(t)$ intersect exactly $k$ segments in $S(t')$ and indices of their rows decreased by $t$ and $t-t'$, respectively, form a shift in $\pi_{M^<}$ of size $k$.
The claim follows.

Consider some $t \in \{0,1,\dots,T-1\}$.
Since no thread succeeds in $\chi$, the thread $Th(t)$ reveals at least $N-n$ blue entries of $A$.
The claim and our assumption $L(\pi_{M^<}) \leq \ell$ imply that the segments from $S(t)$ intersect at most $\ell$ segments from $S(t')$ for every $t' < t$.
Since each segment has length at most $2n$, the thread $Th(t)$ reveals at least $N - 2n - 2tn\ell$ new blue entries of $A$.
This is at least $N/2$ by our assumption $N \geq 4n(\sqrt{n\ell}+1)$ and by the choice of $T$ since
\[N - 2n - 2tn\ell \geq N-2n(T\ell+1) = N - 2n(\sqrt{n\ell}+1) \geq  N/2.\]
Thus, the total number of blue entries in $A$ is at least $TN/2$.
Since the multi-thread scanning visited $n+T$ rows of $A$, there is a vertex $v$ of $K^<_{2N}$ incident to at least $\frac{TN}{2(T+n)}$ blue edges in $\chi$.
Now, our assumption $N \geq 4n(\sqrt{n\ell}+1)$ and the choice of $T$ implies
\[\frac{TN}{2(T+n)} \geq \frac{\sqrt{\frac{n}{\ell}}4n(\sqrt{n\ell}+1)}{2\left(\sqrt{\frac{n}{\ell}}+n\right)} = 2n.\]
Thus, the blue neighborhood of the vertex $v$ contains either a blue triangle or a red copy of $K^<_{2n}$.
This contradicts our assumptions on $\chi$.
\qed
\end{proof}

For every $\varepsilon>0$, Theorem~\ref{thm-upperBoundStronger} immediately implies that $r_<(M^<,K^<_3) \in O(n^{2-\varepsilon})$ for every ordered matching with $\chi_<(M^<)=2$ and $L(\pi_{M^<}) \leq n^{1-2\varepsilon}$.
We show that this is the case for uniform random ordered matchings with interval chromatic number 2 by using the following result by He and Kwan~\cite{kwan20} about the maximum length of a shift in a uniform random permutation on $[2n]$.

\begin{lemma}[\cite{kwan20}]
\label{lem-HeKwan}
A uniform random permutation $\pi$ on $[n]$
satisfies $L\left(\pi\right)\le3\sqrt{n}$ with high probability.
\end{lemma}

Now, it suffices to show that Theorem~\ref{thm-upperBoundStronger} together with Lemma~\ref{lem-HeKwan} implies Theorem~\ref{thm-upperBound}.

\begin{proof}[Proof of Theorem~\ref{thm-upperBound}]
Let $M^<$ be the uniform random ordered matching on $[2n]$.
By Lemma~\ref{lem-HeKwan}, we have $L\left(\pi\right)\le3\sqrt{n}$ with high probability.
Thus, applying Theorem~\ref{thm-upperBoundStronger} with $\ell=3\sqrt{n}$, we obtain 
\[r_<(M^<,K^<_3) \leq 4n(\sqrt{3n^{3/2}} +1) \in O(n^{7/4})\]
with high probability.\qed
\end{proof}

\paragraph{Acknowledgement}

Martin Balko was supported by the grant no.~23-04949X of the Czech Science Foundation (GA\v{C}R) and by the Center for Foundations of Modern Computer Science (Charles University project UNCE/SCI/004). This article is part of a project that has received funding from the European Research Council (ERC) under the European Union's Horizon 2020 research and innovation programme (grant agreement No. 810115).

\bibliography{bibliography}	

\begin{thebibliography}{10}

\bibitem{alon16}
N.~Alon and J.~H. Spencer.
\newblock {\em The probabilistic method}.
\newblock Wiley Series in Discrete Mathematics and Optimization. John Wiley \&
  Sons, Inc., Hoboken, NJ, fourth edition, 2016.

\bibitem{balko}
M.~Balko, J.~Cibulka, K.~Král, and J.~Kynčl.
\newblock Ramsey numbers of ordered graphs.
\newblock {\em Electron. J. Combin.}, 27(1), 2020.

\bibitem{bjv16}
M.~Balko, V.~Jel{\' i}nek, and P.~Valtr.
\newblock On ordered {R}amsey numbers of bounded-degree graphs.
\newblock {\em J. Combin. Theory Ser. B}, 134:179--202, 2019.

\bibitem{cp02}
S.~A. Choudum and B.~Ponnusamy.
\newblock Ordered {R}amsey numbers.
\newblock {\em Discrete Math.}, 247(1-3):79--92, 2002.

\bibitem{crst83}
V.~Chv\'{a}tal, V.~R\"{o}dl, E.~Szemer\'{e}di, and W.~T. Trotter, Jr.
\newblock The {R}amsey number of a graph with bounded maximum degree.
\newblock {\em J. Combin. Theory Ser. B}, 34(3):239--243, 1983.

\bibitem{cibKyn17}
J.~Cibulka and J.~Kyn\v{c}l.
\newblock Better upper bounds on the {F}\"{u}redi-{H}ajnal limits of
  permutations.
\newblock In {\em Proceedings of the {T}wenty-{E}ighth {A}nnual {ACM}-{SIAM}
  {S}ymposium on {D}iscrete {A}lgorithms}, pages 2280--2293. SIAM,
  Philadelphia, PA, 2017.

\bibitem{conlon}
D.~Conlon, J.~Fox, C.~Lee, and B.~Sudakov.
\newblock Ordered {R}amsey numbers.
\newblock {\em J. Combin. Theory Ser. B}, 122:353--383, 2017.

\bibitem{recent_developments}
D.~Conlon, J.~Fox, and B.~Sudakov.
\newblock Recent developments in graph {R}amsey theory.
\newblock In {\em Surveys in combinatorics 2015}, volume 424 of {\em London
  Math. Soc. Lecture Note Ser.}, pages 49--118. Cambridge Univ. Press,
  Cambridge, 2015.

\bibitem{erdosprvni}
P.~Erd\H{o}s and G.~Szekeres.
\newblock A combinatorial problem in geometry.
\newblock {\em Compositio Math.}, 2:463--470, 1935.

\bibitem{kwan20}
X.~He and M.~Kwan.
\newblock Universality of random permutations.
\newblock {\em Bull. Lond. Math. Soc.}, 52(3):515--529, 2020.

\bibitem{odhadtrojuhelnik2}
J.~H. Kim.
\newblock The {R}amsey number {$R(3,t)$} has order of magnitude {$t^2/\log t$}.
\newblock {\em Random Structures Algorithms}, 7(3):173--207, 1995.

\bibitem{msw15}
K.~G. Milans and D.~B. Stolee, D.and~West.
\newblock Ordered {R}amsey theory and track representations of graphs.
\newblock {\em J. Comb.}, 6(4):445--456, 2015.

\bibitem{pachTardos06}
J.~Pach and G.~Tardos.
\newblock Forbidden paths and cycles in ordered graphs and matrices.
\newblock {\em Israel J. Math.}, 155:359--380, 2006.

\bibitem{ramseyprvni}
F.~P. Ramsey.
\newblock On a {P}roblem of {F}ormal {L}ogic.
\newblock {\em Proc. London Math. Soc. (2)}, 30(4):264--286, 1929.

\bibitem{rohatgi}
D.~Rohatgi.
\newblock Off-diagonal ordered {R}amsey numbers of matchings.
\newblock {\em Electron. J. Combin.}, 26(2):Paper No. 2.21, 18, 2019.

\end{thebibliography}
\bibliographystyle{plain}
%\begin{thebibliography}{6}
%%

\end{document}